\documentclass[reqno]{amsart}
\usepackage[margin=1in]{geometry}
\usepackage{xcolor}
\usepackage{enumitem}

\definecolor{bluecite}{HTML}{0875b7}
\usepackage[unicode=true,
bookmarksopen={true},
pdffitwindow=true,
colorlinks=true,
linkcolor=bluecite,
citecolor=bluecite,
urlcolor=bluecite,
hyperfootnotes=false,
pdfstartview={FitH},
pdfpagemode= UseNone]{hyperref}
\usepackage[utf8]{inputenc}
\usepackage[T1]{fontenc}
\newcommand{\diff}{\,\mathrm{d}}
\DeclareMathOperator{\di}{div}

\newtheorem{theorem}{Theorem}[section]

\theoremstyle{definition}

\newtheorem{remark}{Remark}[section]
\numberwithin{equation}{section}

\title[Sharp spectral gaps]{Sharp spectral gap estimates for higher-order operators on Cartan-Hadamard manifolds}
\author{Csaba Farkas}\address{\textsc{Csaba Farkas}: Department of Mathematics and Computer Sciencie, Sapientia Hungarian University of Transylvania, Tg. Mure\c s, Romania}	\email{farkascs@ms.sapientia.ro \& farkas.csaba2008@gmail.com}
\author{S\'andor Kaj\'ant\'o}\address{ \textsc{S\'andor Kaj\'ant\'o}: Department of Mathematics and Computer Science, Babe\c s-Bolyai University, Cluj-Napoca, Romania}
\email{sandor.kajanto@ubbcluj.ro}
\author{Alexandru Krist\'aly}\address{\textsc{Alexandru Krist\'aly}:
Department of Economics, Babe\c s-Bolyai University, Cluj-Napoca, Romania  \& Institute of Applied Mathematics, \'Obuda
University,
Budapest, Hungary
}\email{alex.kristaly@econ.ubbcluj.ro \& kristaly.alexandru@uni-obuda.hu
}
\subjclass[58C40, 35P15, 35R01, 49R05, 58J60
]{58C40, 35P15, 35R01, 49R05, 58J60}
\keywords{Riemannian manifolds, sharp spectral gap, clamped plate, negative curvature.}
\thanks{The authors are supported by the UEFISCDI/CNCS grant PN-III-P4-ID-PCE2020-1001. A. Krist\'aly is also supported by
	Excellence Researcher Program \'OE-KP-2-2022 of \'Obuda University, Budapest, Hungary. Paper dedicated to A.\ Kristály's soccer teammates Anti, Barni, Bidon, Csipi, Endre, Lajos, Lóri, Tasi, Varjú, Vili, Viti and Zoltán.}

\begin{document}
    \begin{abstract} The goal of this paper is to provide sharp spectral gap estimates for problems involving higher-order operators (including both the clamped and buckling plate problems) on Cartan-Hadamard manifolds. The proofs are symmetrization-free -- thus no sharp isoperimetric inequality is needed --  based on two general, yet elementary functional inequalities. The spectral gap estimate for clamped plates solves a sharp asymptotic problem from Cheng and Yang [\textit{Proc.\ Amer.\ Math.\ Soc.}, 2011] concerning the behavior of higher-order eigenvalues on hyperbolic spaces, and answers a question raised in Krist\'aly [\textit{Adv. Math}., 2020] on the validity of such sharp estimates in high-dimensional Cartan-Hadamard manifolds. 
    	 As a byproduct of the general functional inequalities, various Rellich  inequalities are established in the same geometric setting. 
    \end{abstract}

    \maketitle
    % \vspace{-1cm}
%	\tableofcontents
     \vspace{-0.9cm}

    \section{Introduction} \label{sec:intro}
    
In his celebrated book entitled {\it The Theory of Sound}, Lord Rayleigh \cite{L-R} formulated various questions concerning the qualitative behavior of the first eigenvalue for \textit{fixed membrane}, \textit{clamped plate} and \textit{buckling plate} problems. Although these problems have been posed for domains in the Euclidean setting, the mathematical community started to study them not only within linear structures, but also on curved spaces.  

Due to its second order character, the \textit{fixed membrane problem} turned out to be the most accessible among the aforementioned problems, which can  be written as  
\begin{equation}\label{eq:fmp}
	\begin{cases}
		\Delta_g u = -\lambda_\text{m} u, & \text{in } \Omega,\\
		u=0,& \text{on }\partial \Omega,
	\end{cases}
\end{equation}
where $\Omega$ is an open bounded subset of an \(n(\geq 2)\)-dimensional Riemannian manifold $(M,g)$ and \(\Delta_g\) stands for the \emph{Laplace-Beltrami operator} on $(M,g)$. In the particular case when $(M,g)=(\mathbb R^n,g_0)$ is the standard Euclidean space, Faber \cite{Faber} and Krahn \cite{Krahn} proved that the first eigenvalue of \eqref{eq:fmp} is not smaller than the value $j_{n/2-1,1}^2(\omega_n/{\rm Vol}(\Omega))^{2/n}$, where $j_{\mu,1}$ is the first positive zero of the Bessel function $J_\mu$ of the first kind with order $\mu$, $\omega_n$ is the volume of the unit ball in $\mathbb R^n$, and ${\rm Vol}(\Omega)$ is the volume of $\Omega$; moreover, equality is achieved whenever $\Omega$ is a ball and the eigenvalues for larger and larger balls tend to zero. The crucial step in the proof of Faber-Krahn's result is the P\'olya-Szeg\H o inequality, which is based on  Schwarz symmetrization and the sharp isoperimetric inequality in $\mathbb R^n$. Their proof can be easily extended to any  \textit{Cartan-Hadamard manifold} (complete, simply connected Riemannian manifold with nonpositive sectional curvature) which satisfies the so-called \textit{Cartan-Hadamard conjecture}. The latter conjecture is nothing but the sharp isoperimetric inequality on Cartan-Hadamard manifolds, formally being the same as its classical, Euclidean counterpart; we note that this conjecture is confirmed only in low dimensions $n\in \{2,3,4\}$.     

%In particular, \eqref{eq:fmpvar} implies that the first eigenvalue of the \textit{fixed membrane} problem 

One of the most surprising facts in spectral theory on Riemannian manifolds is due to  McKean \cite{mckean1970upper}, which roughly states that strong negative curvature produces a universal, domain-independent spectral gap for the first/principal eigenvalue of \eqref{eq:fmp}, which is in  radical contrast with the Euclidean case. More precisely, if the sectional curvature satisfies \({\bf K}\le-\kappa^2\) for some \(\kappa> 0\) on a Cartan-Hadamard manifold $(M,g)$,  then 
\begin{equation}\label{eq:fmpvar}
\lambda_\text{m}(\Omega):=	\inf_{u\in C_0^\infty(\Omega)\setminus \{0\}} \frac{\displaystyle\int_\Omega |\nabla_gu|^2\diff v_g}{\displaystyle\int_\Omega u^2\diff v_g}\ge  \frac{(n-1)^2\kappa^2}{4},
\end{equation} 
for every open bounded subset $\Omega\subset M$, 
where \(\nabla_g\) and \(\diff v_g\) denote the \emph{Riemannian gradient} and the \emph{canonical volume form} on $(M,g)$, respectively.

%  To be more precise,   
%is always controlled from below by the universal constant $\frac{(n-1)^2\kappa^2}{4}$, independently on the domain $\Omega$; here, 
%    Let \((M,g)\) be a complete, non-compact \(n\)-dimensional Riemannian manifold with {\color{red}\(n\ge 2\)}. Suppose that the sectional curvature satisfies \({\bf K}\le-\kappa^2\), for some \(\kappa\ge0\). Let \(\Omega\subset M\) be a bounded open  domain and consider the \emph{fixed membrane} problem:
%   A fundamental result in the spectral theory of negatively curved spaces is the celebrated estimate by McKean \cite{mckean1970upper}, which states that for every \(u\in C_0^\infty(\Omega)\) one has
%    \begin{equation}
%        \int_\Omega |\nabla_gu|^2\diff v_g\ge  \frac{(n-1)^2\kappa^2}{4} \int_\Omega |u|^2\diff v_g,\label{eq:fmpvar}
%    \end{equation}
%     It is worth to highlight that the constant is independent of volume of the domain, it only depends on the dimension and the curvature bound. 
     Moreover,  the bound in \eqref{eq:fmpvar} is \emph{sharp}; indeed, if we consider the model space form  \(M={\bf M}^n_{-\kappa^2}\) of constant sectional curvature \({\bf K}=-\kappa^2\), and the ball \(\Omega=B_R\subset {\bf M}^n_{-\kappa^2}\)  with radius \(R\), then the first eigenvalue of \eqref{eq:fmp} has the  limiting property
    \[
        \lim_{R\to\infty}\lambda_\text{m}(B_R)=    \frac{(n-1)^2\kappa^2}{4}, 
    \]
    see e.g.\ Chavel \cite{Chavel} and further asymptotically improved versions by Borisov and Freitas \cite{B-F},   Cheng and Yang ~\cite{ChengYang_JDE},  Krist\'aly \cite{Kristaly-ACV}, Savo \cite{Savo}, and references therein.   A natural extension of the above results is the $p$-\textit{fixed membrane problem}, which   can be obtained by replacing the PDE from \eqref{eq:fmp}  by  
    $\Delta_{g,p}=-\lambda_{\text{m},p}|u|^{p-2}u,$
    where \(p>1\) and \(\Delta_{g,p}\) denotes the \emph{\(p\)-Laplace-Beltrami operator}. In this case the sharp estimate reads as 
    \begin{equation}\label{eq:fmpvar:p}
    	\int_\Omega |\nabla_gu|^p\diff v_g\ge  \frac{(n-1)^p\kappa^p}{p^p} \int_\Omega |u|^p\diff v_g,\ \ \forall u\in C_0^\infty(\Omega),
    \end{equation} 
    see He and Yin \cite{yin2014first} and  Kajántó, Kristály, Peter and Zhao \cite{ricattipair2023} for an alternative proof.

  The \emph{clamped plate problem} is definitely more sophisticated than the  fixed membrane problem,  coming from its fourth order character, which is formulated as 
   \begin{equation}\label{eq:cpp}
  	\begin{cases}
  		\Delta_{g}^2u=\lambda_{\text{c}}u, & \text{in } \Omega,\\
  		u=\frac{\partial u}{\partial \bf n}=0,& \text{on }\partial \Omega,
  	\end{cases}
  \end{equation}
 where  \(\Delta_{g}^2\) denotes the \emph{biharmonic operator}, and \(\frac{\partial}{\partial \bf n}\) stands for the outward pointing \emph{normal derivative}. Dealing with Lord Rayleigh's initial conjecture in the Euclidean case $(M,g)=(\mathbb R^n,g_0)$, Ashbaugh and Benguria \cite{A-B} and Nadirashvili \cite{Nadir} stated the sharp  Faber-Krahn-type inequality 
in dimensions 2 and 3, proving that the first eigenvalue for \eqref{eq:cpp} is controlled below by $h_{n/2-1}^4(\omega_n/{\rm Vol}(\Omega))^{4/n}$, where $h_\mu$ is the first positive zero of the cross product of the Bessel functions $J_\mu$ and $I_\mu$. Similarly as in the fixed membrane case, larger and larger domains produce smaller and smaller first eigenvalues, which tend to zero; for a quantitative form, see Antunes,  Buoso and  Freitas \cite{ABF}. 

Clamped plate problems have been recently studied on Riemannian manifolds, both for positively and negatively curved spaces, see Krist\'aly \cite{Kristaly-GAFA, KristalyAdv}. In particular, in Cartan-Hadamard manifolds  with sectional curvature satisfying \({\bf K}\le-\kappa^2\) for some \(\kappa>0\), the author proved a higher-order form of McKean's spectral gap estimate; namely, one has that
 \begin{equation}\label{eq:fmpvar-0}
 	\lambda_\text{c}(\Omega):=	\inf_{u\in C_0^\infty(\Omega)\setminus \{0\}} \frac{\displaystyle\int_\Omega (\Delta_g u)^2\diff v_g}{\displaystyle\int_\Omega u^2\diff v_g}\ge  \frac{(n-1)^4\kappa^4}{16},
 \end{equation} 
whenever the \emph{\(\kappa\)-Cartan-Hadamard conjecture}  holds, see \cite[Theorem 1.1]{KristalyAdv}.  This conjecture is valid for general Cartan-Hadamard  manifolds in dimension \(n\in\{2,3\}\), see Bol \cite{bol1941isoperimetrische} and Kleiner \cite{kleiner1992isoperimetric}, and for space forms \(M={\bf M}^n_{-\kappa^2}\)  in any dimension, see Dinghas \cite{dingas1949}. The proof of \eqref{eq:fmpvar-0}  deeply relies on Schwarz-type symmetrization and the validity of the aforementioned  \emph{\(\kappa\)-Cartan-Hadamard conjecture}, which is the strong  \emph{\(\kappa\)-sharp isoperimetric inequality}; for a detailed discussion,  see  Kloeckner and Kuperberg \cite{Klockner}. 

Our first result, based on a symmetrization-free approach, reads as follows: 

 \begin{theorem}\label{thm:cpp:intro}
	Let \((M,g)\) be an   \(n\)-dimensional Cartan-Hadamard manifold with \(n\ge2\) and assume  that the sectional curvature satisfies \({\bf K}\le-\kappa^2\) for some \(\kappa>0\). Let  \(p>1\) and any domain \(\Omega\subset M\). Then for every \(u\in C_0^\infty(\Omega)\)  one has
	\begin{equation}\label{eq:cppvar}
		\int_\Omega |\Delta_g u|^p\diff v_g\ge \frac{(n-1)^{2p}\kappa^{2p}(p-1)^p}{p^{2p}}\int_\Omega  |u|^p\diff v_g.
	\end{equation} 
Moreover, the constant in \eqref{eq:cppvar} is sharp. 
\end{theorem}

We notice that \eqref{eq:cppvar} is known on the hyperbolic space \(M={\bf M}^n_{-\kappa^2}\) by Ng{\^o} and Nguyen \cite{ngo2019sharp}. In the latter paper, the authors deeply explore symmetrization techniques combined with the validity of the  \(\kappa\)-sharp isoperimetric inequality, where the model structure of \(M={\bf M}^n_{-\kappa^2}\) plays a crucial role. Note, however, that in Theorem \ref{thm:cpp:intro}  we deal with \textit{generic} Cartan-Hadamard manifolds, and no symmetrization is applied. In fact, the proof is based on a \textit{general functional inequality} (see Theorem \ref{thm:mainu}), making connection  between \(|\Delta_g u|^p\) and \(|u|^p\) for general \(p>1\),
 whose proof uses only the divergence theorem, a Laplace comparison and the convexity of the function $ |\cdot|^p$ with $p>1$. 

In addition, Theorem \ref{thm:cpp:intro} extends not only the validity of \eqref{eq:fmpvar-0} to any dimension (this estimate being proved only in dimensions 2 and 3, cf.\  Krist\'aly \cite{KristalyAdv}), but also solves the claim raised in  Cheng and Yang \cite{ChengYangPAMS} and Li, Jing and Zeng \cite{LiJingZeng}. In fact, in the latter two works the authors proved that \textit{if} the first eigenvalue of the clamped plate problem satisfies
\begin{equation}\label{eq:limcond}
	\lim_{R\to\infty}\lambda_{\text{c}}(B_R)=\frac{(n-1)^4}{16},    
\end{equation}
where $B_R$ is the geodesic ball in the hyperbolic space ${\bf M}^n_{-1}$, then the same limit should be valid also for the $l^{\rm th}$ eigenvalues of \eqref{eq:cpp}, $l\geq 2.$ Now, in view of Theorem \ref{thm:cpp:intro}, the assumption  \eqref{eq:limcond} in   \cite{ChengYangPAMS,LiJingZeng} turns out to be superfluous.

%deed, the  the estimate \eqref{eq:cppvar} is proved when \(M={\bf M}^n_{-1}\) and \(p=2\). Additionally, they showed that 
%then every eigenvalue has this property, as well. 

For the \emph{buckling plate problem}, which can be states as
    \begin{equation}\label{eq:bpp}
        \begin{cases}
            \Delta_{g}^2u=-\lambda_{\text{b}}\Delta_gu, & \text{in } \Omega,\\
            u=\frac{\partial u}{\partial \bf n}=0,& \text{on }\partial \Omega,
        \end{cases}
    \end{equation}
only a few qualitative information are known in the geometric setting; however, our second result states  a sharp spectral gap on generic Cartan-Hadamard manifolds: 
   
    \begin{theorem}\label{thm:bpp:intro}
        Let \((M,g)\) be an   \(n\)-dimensional Cartan-Hadamard manifold with \(n\ge2\) and assume  that the sectional curvature satisfies \({\bf K}\le-\kappa^2\) for some \(\kappa>0\). If \(\Omega\subseteq M\) is any  domain, then for every \(u\in C_0^\infty(\Omega)\) one has 
        \begin{equation}\label{eq:bppvar}
            \int_\Omega |\Delta_g u|^2\diff v_g\ge \frac{(n-1)^2\kappa^2}{4}\int_\Omega |\nabla_g u|^2\diff v_g.
        \end{equation}
    Moreover, the constant in \eqref{eq:bppvar} is sharp.  
    \end{theorem}

Inequality \eqref{eq:bppvar} is again known in model hyperbolic spaces (even for $p>1$), see Ng{\^o} and Nguyen \cite{ngo2019sharp}, where the aforementioned symmetrization techniques are applied with the sharp isoperimetric inequality. The proof of Theorem \ref{thm:bpp:intro} is carried out by a second \textit{general functional inequality} (see Theorem \ref{thm:maingru}) that makes connection between \(|\Delta_g u|^2\) and \(|\nabla_g u|^2\),  based again on the divergence theorem, Laplace comparison and some convexity arguments. Note, however, that  certain technical difficulties prevent the extension of this functional inequality to the general case $p>1.$

    The sharpness of both constants in \eqref{eq:cppvar} and \eqref{eq:bppvar}  can be established in the usual way, by constructing suitable  sequences of functions in the model space whose limits  provide the sharp constants. Furthermore, if we apply iteratively either  \eqref{eq:cppvar} or  \eqref{eq:bppvar}  for functions \(u\), \(\Delta_g u\), \(\Delta_g^2 u\), \ldots\ and combine the results with inequality \eqref{eq:fmpvar:p}, we obtain higher-order sharp spectral gap estimates; see Theorem \ref{thm:ho:cp} \& \ref{thm:ho:bp}.

    The proofs of our main Theorems \ref{thm:cpp:intro} \& \ref{thm:bpp:intro} easily follow by choosing  constant test functions as the parameter functions in the general functional inequalities (see Theorems \ref{thm:mainu} \& \ref{thm:maingru}). However, choosing different parameter functions  in Theorems \ref{thm:mainu} \& \ref{thm:maingru}, as a byproduct, we obtain simple alternative proofs of classical and weighted Rellich inequalities, as well as their higher-order versions on Cartan-Hadamard  manifolds. As we already noticed, these proofs do not require the validity of any isoperimetric inequality; see Theorem \ref{thm:Rellich:weighted} \& \ref{thm:Rellich:ho}. Finally, by considering more sophisticated parameter functions, we provide elegant proofs to some  Rellich-type inequalities; see Theorem \ref{thm:impr:1}-\ref{thm:impr:3}.

    The paper is structured as follows. In Section \ref{sec:prelim} we recall some preliminary notions and results. In Section~\ref{sec:gfi} we present the two general functional inequalities. In Section~\ref{sec:spectralgaps} we prove the sharp spectral gap estimates from Theorem \ref{thm:cpp:intro} \& \ref{thm:bpp:intro}, and their higher-order variants. In Section~\ref{sec:Rellich} we give a short, alternative proof for the classical and weighted Rellich-type inequality and their higher-order versions. Additionally, we provide short proofs for some well-known Rellich-type inequalities.

    \section{Preliminaries}\label{sec:prelim} 
    In this section we recall some preliminary definitions and results; we mainly follow Gallot,  Hulin and  Lafontaine \cite{gallot2004riemmanian} and Hebey \cite{hebey2000nonlinear}. Let \((M,g)\) be an \(n\)-dimensional Riemannian manifold, with \(n\ge 2\). Let \(p>1\) and \(u\in C^\infty_0(M)\) be a compactly supported smooth function. Let \((x^i)\) be a local coordinate system in the coordinate neighborhood of \(x\in M\). The gradient of \(u\) is \(\nabla_g u\), having components 
    \[
        u^i=g^{ij}\frac{\partial u}{\partial x^j},
    \]
    while the usual Laplace-Beltrami operator is $\Delta_{g}u={\rm div}_g(\nabla_g u)$. 
%    where \(g^{ij}\) are the components of \(g^{-1}\). The \emph{\(p\)-Laplace-Beltrami operator} is given by 
%    \[
%        \Delta_{g,p}u=\di_g(|\nabla_g u|^{p-2}\nabla_g u).  
%    \]
%    If \(p=2\), then \(\Delta_{g,2}\) is precisely \(\Delta_g\), the usual Laplace-Beltrami operator. The \emph{\(p\)-biharmonic operator} is defined by
%    \[
%        \Delta_{g,p}^2u=\Delta_g(|\Delta_g u|^{p-2}\Delta_g u).    
%    \]
%    If \(p=2\), then \(\Delta_{g,2}^2\) is exactly \(\Delta_g^2\), the usual biharmonic operator. 

    If \(u,v\in C_0^2(M)\) then we have the following identities 
    \[
        \int_M u\Delta_g v\diff v_g = -\int_M \nabla_g u\nabla_g v\diff v_g\quad\mbox{and}\quad \int_M u\Delta_g v\diff v_g =\int_M v\Delta_g u\diff v_g,
    \]
    referred as \emph{integration by parts} and \emph{Green's second identity}, respectively.  

    We use the notation \(d_g(x,y)\) for the \emph{Riemannian distance} between \(x,y\in M\). For a fixed \(x_0\in M\) we denote \(d_{x_0}(x)=d_g(x_0,x)\) the distance from \(x_0\).  The \emph{eikonal equation} states that  \(\diff v_g\)-a.e. on \(M\), one has
    \begin{equation}\label{eq:eikonal}
        |\nabla_g d_{x_0}|=1.
    \end{equation}

    For \(\kappa\ge 0\), the model space form ${\bf M}_{-\kappa^2}^n$ is an $n$-dimensional Riemannian manifold with constant sectional curvature ${\bf K}=-\kappa^2$; more precisely
    \[
        {\bf M}_{-\kappa^2}^n = \begin{cases}
            \mathbb{R}^n\mbox{ -- the Euclidean space}, & \mbox{if }\kappa = 0,\\
            \mathbb H^n_{-\kappa^2}\mbox { -- the Hyperbolic space}, & \mbox{if } \kappa>0.
        \end{cases}
    \]

    Define the function \({\bf ct}_\kappa(t)\colon(0,\infty)\to(0,\infty)\) by
    \[
        {\bf ct}_\kappa(t)=\begin{cases}
            \frac{1}{t},&\text{if } \kappa=0,\\
            \kappa\coth(\kappa t),&\text{if } \kappa>0.
        \end{cases}
    \]
    The following Laplace comparison principle holds, see e.g.\ \cite[Theorem 3.101]{gallot2004riemmanian}.
    \begin{theorem}\label{comparison-theorem} Let $(M,g)$ be an $n$-dimensional, complete  Riemannian manifold, with $n\geq 2$. Fix $x_0\in M$ and suppose that the sectional curvature satisfies \({\bf K}\le -\kappa^2\) for some \(\kappa\ge 0\).  Then one has
    \[
        \Delta_g d_{x_0}\ge (n-1)  {\bf ct}_\kappa(d_{x_0}).
    \]
        Moreover, equality holds if and only if $(M,g)$ is isometric to the model space form ${\bf M}_{-\kappa^2}^n$.
    \end{theorem}

    \section{General functional inequalities}\label{sec:gfi}
    In this section we present two general functional inequalities. The first inequality connects \(|\Delta_g u|^p\) and \(|u|^p\) for \(p>1\) and it is tailored to provide sharp spectral gap estimate for the clamped plate problem \eqref{eq:cpp}, even for general $p>1$. The second inequality connects \(|\Delta_g u|^2\) and \(|\nabla_g u|^2\), and it is designed to provide sharp spectral gap estimate for the buckling plate problem \eqref{eq:bpp}. 
    
    The first inequality can be stated as follows.
    \begin{theorem}\label{thm:mainu}
        Let \((M,g)\) be an \(n\)-dimensional, complete, non-compact Riemannian manifold, with \(n\ge 2\). Let \(\Omega\subseteq M\) be a domain, \(x_0\in\Omega\) and \(\rho=d_{x_0}\). Let \(p>1\) and suppose that \(L,W,w,G,H\colon(0,\sup\rho)\to(0,\infty) \) satisfy the following conditions:
        \begin{enumerate}[label=\rm({\bf C\arabic*})]
            \item\label{cond:c1} \(L,W\) are continuous, \(w,G\) are of class \(C^2\) and \(H\) is of class \(C^1\);
            \item\label{cond:c2} \(\Delta_g \rho\ge L(\rho)\) in the distributional sense, and \((wG)'\le0\);
            \item\label{cond:c3} the ordinary differential inequality
            \begin{equation}
                (p-1)\left[2(wGH)'+2wGHL-pwGH^2-w|G|^{p'}\right] -(wG)''-(wG)'L \ge W\label{eq:odigen}
            \end{equation}
            holds for the functions \(L(t),W(t),w(t),G(t),H(t)\), for all \(t\in(0,\sup\rho)\).
        \end{enumerate}
        Then for every \(u\in C_0^\infty(\Omega)\) one has
        \[
            \int_\Omega w(\rho)  |\Delta_g u|^p\diff v_g \ge \int_\Omega W(\rho)|u|^p\diff v_g.
        \]
    \end{theorem}
    \begin{proof}
        The convexity of $\xi\mapsto|\xi|^{p}$ implies 
        \begin{equation}\label{eq:conv:proof1}
            |\xi|^{p}\geq |\eta|^{p} + p  |\eta|^{p-2} (\xi-\eta)\eta=p|\eta|^{p-2}\xi\eta+(1-p)|\eta|^{p},\qquad\forall \xi,\eta,
        \end{equation}
        where both \(\xi\) and \(\eta\) are either scalars or vectors of the same type.
        Fix $u\in C_{0}^\infty(\Omega)$ arbitrarily. Choose \[\xi=\Delta_gu\quad\mbox{and}\quad\eta=-|G(\rho)|^\frac{2-p}{p-1}G(\rho)u\] to obtain
        \[
            |\Delta_g u|^p\ge -p G(\rho) |u|^{p-2} u\Delta_g u+(1-p)|G(\rho)|^{p'}|u|^p.   
        \]
        Multiplying both sides by \(w(\rho)>0\) and integrating over \(\Omega\) yields
        \[
            \int_\Omega w(\rho)  |\Delta_g u|^p\diff v_g 
            \ge -p\int_\Omega w(\rho) G(\rho) |u|^{p-2} u\Delta_g u\diff v_g +(1-p)\int_\Omega w(\rho) |G(\rho)|^{p'}|u|^p\diff v_g.
        \]
      
       We shall focus on the second term. 
        On the one hand, using the relation
        \[
            \Delta_g \frac{|u|^p}{p}=(p-1)|u|^{p-2}|\nabla_g u|^2+|u|^{p-2}u\Delta_g u,
        \]
        and Green's second identity leads us to
        \[
            -p\int_\Omega w(\rho) G(\rho) |u|^{p-2} u\Delta_g u\diff v_g=p(p-1)\int_\Omega w(\rho)G(\rho)|u|^{p-2}|\nabla_g u|^2\diff v_g-\int_\Omega|u|^p\Delta[w(\rho)G(\rho)]\diff v_g.
        \]

        On the other hand, choosing \(p=2\), \(\xi=\nabla_g  u\) and \(\eta=-uH(\rho) \nabla_g\rho\) in inequality \eqref{eq:conv:proof1} implies
        \[
            |\nabla_g  u|^2\ge -2H(\rho)u\nabla_g  u\nabla_g  \rho-H(\rho)^2|u|^2,
        \]
        hence
        \begin{align*}
            \int_\Omega w(\rho)G(\rho)|u|^{p-2}|\nabla_g u|^2\diff v_g
            &\ge -2\int_\Omega w(\rho)G(\rho)H(\rho)|u|^{p-2}u\nabla_g  u\nabla_g  \rho\diff v_g-\int_\Omega w(\rho)G(\rho)H(\rho)^2|u|^p\diff v_g.
        \end{align*}
 Finally, an integration by parts yields 
        \begin{align*}
            -2\int_\Omega w(\rho)G(\rho)H(\rho)|u|^{p-2}u\nabla_g  u\nabla_g  \rho \diff v_g&=-\frac{2}{p}\int_\Omega w(\rho)G(\rho)H(\rho)\nabla_g |u|^p\nabla_g \diff v_g \rho\\
            &=\frac{2}{p}\int_\Omega|u|^p\di_g(w(\rho)G(\rho)H(\rho)\nabla_g \rho)\diff v_g.
        \end{align*}

        By the above computations, for every \(u\in C_0^\infty(\Omega)\) one has
        \[
            \int_\Omega w(\rho)  |\Delta_g u|^p\diff v_g \ge \int_\Omega |u|^pW(\rho)\diff v_g,
        \]
        provided that
        \begin{align*}
            W(\rho)&\le(p-1)(2\di_g(w(\rho)G(\rho)H(\rho)\nabla_g\rho)-p w(\rho)G(\rho)H(\rho)^2-w(\rho)|G(\rho)|^{p'})-\Delta(w(\rho)G(\rho))\\
            &\le(p-1)\left[2(w(\rho)G(\rho)H(\rho))'+2w(\rho)G(\rho)H(\rho)\Delta_g\rho-pw(\rho)G(\rho)H(\rho)^2-w(\rho)|G(\rho)|^{p'}\right]\\
            &\qquad -[w(\rho)G(\rho)]''-[w(\rho)G(\rho)]'\Delta_g \rho,
        \end{align*}
        which easily follows by \ref{cond:c2} and \ref{cond:c3}.
    \end{proof}

\begin{remark}\rm Recently, the concept of Riccati-pairs for certain weights has been introduced by Kajántó, Kristály, Peter and Zhao \cite{ricattipair2023} in order to establish sharp Hardy-type inequalities,  similar to the Bessel-pairs defined by Ghoussoub and Moradifam \cite{GM}.    Condition \eqref{eq:odigen} can be viewed as a higher order Riccati-type ordinary differential inequality which is crucial to prove functional inequalities involving the terms \(|\Delta_g u|^p\) and \(|u|^p\).  In the same spirit, inequality \eqref{eq:pdigen} in the forthcoming  Theorem \ref{thm:maingru} plays a similar role for proving  functional inequalities involving the terms  \(|\Delta_g u|^2\) and \(|\nabla_g u|^2\). 
\end{remark}

    For simplicity, we state the second functional inequality in unweighted form as follows.
    
    \begin{theorem}\label{thm:maingru} Let \((M,g)\) be an \(n\)-dimensional complete, non-compact Riemannian manifold, with \(n\ge 2\). Let \(\Omega\subseteq M\) be a domain, \(x_0\in\Omega\) and \(\rho=d_{x_0}\). Suppose that \(L,W,G,H\colon(0,\sup\rho)\to(0,\infty) \) satisfy the following conditions:
        \begin{enumerate}[label=\rm({\bf C\arabic*'})]
            \item\label{cond:c1p} \(L,W\) are continuous, \(G\) is of class \(C^2\) and \(H\) is of class \(C^1\);
            \item\label{cond:c2p} \(\Delta_g \rho\ge L(\rho)\) in the distributional sense;
            \item\label{cond:c3p} the partial differential inequality 
            \begin{equation}
                (W(\rho)H(\rho))'+W(\rho)H(\rho)L(\rho)-W(\rho)H(\rho)^2\ge \Delta_g G(\rho)+G(\rho)^2,\label{eq:pdigen}
            \end{equation}
            holds for \(\rho=d_{x_0}(x)\), for all \(x\in\Omega\).
        \end{enumerate}
        Then for every \(u\in C_0^\infty(\Omega)\) one has
        \[
            \int_\Omega |\Delta_g u|^2\diff v_g \ge \int_\Omega (2G(\rho)-W(\rho))|\nabla_g u|^2\diff v_g.
        \]
    \end{theorem}
    \begin{proof} 
        For \(p=2\), the convexity inequality \eqref{eq:conv:proof1} reads as 
        \begin{equation}\label{eq:conv:proof2}
            \xi^2\ge 2\xi\eta -\eta^2,\quad\forall \xi,\eta.
        \end{equation}
        If \(\xi=\Delta_g u\) and \(\eta=-G(\rho)u\), one has that         \[
            (\Delta_g u)^2 \ge -2 G(\rho) u\Delta_g u - G(\rho)^2u^2.  
        \] 
        Integrating over \(\Omega\) yields
        \[
            \int_\Omega (\Delta_g u)^2\diff v_g \ge -2  \int_\Omega G(\rho) u\Delta_g u\diff v_g -  \int_\Omega G(\rho)^2u^2\diff v_g.   
        \]
        By using relation \(-2u\Delta_g u=2|\nabla_g u|^2-\Delta_g (u^2)\) and Green's second identity in the second term leads us to
        \[
            \int_\Omega (\Delta_g u)^2\diff v_g \ge  2\int_\Omega G(\rho)|\nabla_g u|^2\diff v_g-\int_\Omega(\Delta_g G(\rho))u^2 \diff v_g -  \int_\Omega G(\rho)^2u^2\diff v_g. 
        \]
        To finish our proof it is enough to show that 
        \begin{equation}\label{eq:improofgrad}
            \int_\Omega W(\rho)|\nabla_g u|^2\diff v_g  \ge \int_\Omega(\Delta_g G(\rho)+G(\rho)^2)u^2 \diff v_g.
        \end{equation}
        Choosing \(\xi=\nabla_g u\) and \(\eta=-u H(\rho)\nabla_g\rho\) in inequality \eqref{eq:conv:proof2}, we infer that  
        \[
            |\nabla_g u|^2\ge -2H(\rho)u\nabla_g u\nabla_g \rho- u^2H(\rho)^2.   
        \]
        Multiplying both sides with \(W(\rho)\) and integrating over \(\Omega\) yields
        \[
            \int_\Omega W(\rho) |\nabla_g u|^2\diff v_g\ge -2\int_\Omega W(\rho) H(\rho)u\nabla_g u\nabla_g \rho\diff v_g-  \int_\Omega W(\rho) H(\rho)^2u^2\diff v_g.   
        \]
        An integration by parts and condition \ref{cond:c2p} implies that 
        \begin{align*}
            -2\int_\Omega W(\rho) H(\rho)u\nabla_g u\nabla_g \rho\diff v_g&= -\int_\Omega W(\rho)H(\rho)\nabla_g\rho\nabla_g(u^2)\diff v_g=\int_\Omega\di_g[W(\rho)H(\rho)\nabla_g\rho]u^2\diff v_g\\
            &=\int_\Omega [W'(\rho)H(\rho)+W(\rho)H'(\rho)+W(\rho)H(\rho)\Delta_g\rho]u^2\diff v_g\\
            &\ge \int_\Omega [W'(\rho)H(\rho)+W(\rho)H'(\rho)+W(\rho)H(\rho)L(\rho)]u^2\diff v_g.
        \end{align*}
        Finally condition \ref{cond:c3p} yields \eqref{eq:improofgrad}, concluding the proof.
    \end{proof}
    \begin{remark} Several comments are in order.
        \begin{itemize}[leftmargin=0.5cm]
            \item[a)] Compare \ref{cond:c3} and \ref{cond:c3p} to observe that the first condition involves an ordinary differential inequality, while the second involves a partial differential inequality on the manifold; the latter is  due to the dependence of \(\Delta_g\) on \(\rho\). For a radial function \(G(\rho)\) one has \(\Delta_g G(\rho)=G''(\rho)+G'(\rho)\Delta_g\rho\). Hence \ref{cond:c3p} is genuinely harder to verify than \ref{cond:c3}. However, when \(G\) is constant, \(\rho\) can be simply replaced with a scalar \(t\), and the partial differential inequality reduces to  an ordinary one.
            \item[b)] The technique presented in the proof of Theorem \ref{thm:maingru} only works for \(p=2\). For general \(p>1\) the second term of the convexity inequality \eqref{eq:conv:proof1} contains \(|u|^{p-2}u\Delta_g u\) which can not be transformed into \(|\nabla_g u|^p\).
            \item[c)] We are not aware of any 
           simple convexity arguments (with arbitrary choices of \(\xi\) and \(\eta\)) and possibly multiple uses of integration by parts which could provide a general functional inequality involving integrals of \(|\Delta_g u|^p\) and \(|\nabla_g u|^p\) for general \(p>1\).
        \end{itemize}
    \end{remark}

    \section{Sharp spectral gap estimates}\label{sec:spectralgaps} In this section we prove Theorems \ref{thm:cpp:intro} \& \ref{thm:bpp:intro}. By an iterative applications of these results and using inequality \eqref{eq:fmpvar} we get higher-order estimates as well.
    \subsection{Clamped plate problem: proof of Theorem \ref{thm:cpp:intro}}\label{ssec:clamped} 
%    The sharp spectral gap estimate for the clamped plate problem is presented in Theorem~\ref{thm:cpp:intro}. For better readability we repeat the statement here, before the proof.
%    \begin{theorem}\label{thm:cpp}
%        Let \((M,g)\) be an \(n\)-dimensional, complete, non-compact Riemannian manifold, with \(n\ge 2\). Suppose that the sectional curvature satisfies \({\bf K}\le\kappa
%        ^2\), for some \(\kappa>0\). Let \(\Omega\subseteq M\) be a domain and \(p>1\). Then for every \(u\in C_0^\infty(\Omega)\) one has
%        \begin{equation}
%            \int_\Omega |\Delta_g u|^p\diff v_g\ge \left(\frac{(n-1)^2(p-1)\kappa^2}{p^2}\right)^p\int_\Omega|u|^p.
%            \label{eq:mcKeanu}
%        \end{equation}
%        Moreover the constant is sharp.
%    \end{theorem}
        In Theorem \ref{thm:mainu} let us choose
        \[
            \quad L\equiv(n-1)\kappa,\quad W\equiv C,\quad  w\equiv 1,\quad G\equiv a\quad\mbox{and}\quad H\equiv b,  
        \]
        for some constants \(C,a,b>0\), which will be determined later. Condition \ref{cond:c1} clearly holds. By Laplace comparison (see Theorem \ref{comparison-theorem}) one has 
        \[\Delta_g \rho\ge (n-1)\kappa\coth(\kappa \rho)\ge (n-1)\kappa.\]
        Additionally, since \(G\) is constant, condition \ref{cond:c2} holds as well. Inequality \eqref{eq:odigen} from condition \ref{cond:c3} is equivalent to 
        \[
            f(a,b):=(p-1) \left(2ab \kappa 
            (n-1)-a^{\frac{p}{p-1}}-a b^2p\right)\ge C.
        \] 
        The best choice for the constant $C$ is obtained for $$\max_{a,b}f(a,b)=f\left(\left(\frac{(n-1)^2(p-1)\kappa^2}{p^2}\right)^{p-1},\frac{(n-1)\kappa}{p}\right)=\left(\frac{(n-1)^2(p-1)\kappa^2}{p^2}\right)^p.$$ 
%        Let us denote  \[\widetilde{a}=\left(\frac{(n-1)^2(p-1)\kappa^2}{p^2}\right)^{p-1}\quad\mbox{and}\quad\widetilde{b}=\frac{(n-1)\kappa}{p}.\]
%        A simple computation yields that  \(\max_{a,b}f(a,b)\) is achieved for the pair $(\widetilde{a},\widetilde{b})$,  hence
%        \[
%            f(\widetilde{a},\widetilde{b})=\left(\frac{(n-1)^2(p-1)\kappa^2}{p^2}\right)^p\ge C.
%        \]
           
        To prove the sharpness, fix \(\delta>0\) and define the \emph{truncation function} 
        \begin{equation}\label{eq:truncation:1}
            \phi(t)=\begin{cases}
                t-\frac{\delta}{2},&\mbox{if } t\in\left[\frac{\delta}{2},\frac{\delta}{2}+1\right],\\
                1,&\mbox{if } t\in\left[\frac{\delta}{2}+1,\delta-1\right],\\
                \delta-t,&\mbox{if } t\in\left[\delta-1,\delta\right],\\
                0,& \mbox{otherwise}.
            \end{cases}
        \end{equation}
        Let \(s=\frac{(n-1)\kappa}{p}\)
        and choose 
        $
            u_\delta=\phi(\rho) e^{-s\rho} $ on $\Omega={\bf M}_{-\kappa^2}^n.    
        $
        Due to the definition of \(\rho\),   we have 
        \[
            |\nabla_g\rho|=1\quad\mbox{and}\quad\Delta_g\rho=(n-1)\kappa\coth(\kappa \rho)=ps\coth(\kappa\rho).
        \]
        By using the fact that \(\phi''=0\) (except a finite number of points), one has ${\rm d}v_g$-a.e.\ that
        \begin{align*}
            \nabla_g u_\delta&=(\phi'(\rho)-s\phi(\rho))e^{-s\rho}\nabla_g\rho,\\
            \Delta_g u_\delta&=\left[-2s\phi'(\rho)+s^2\phi(\rho)+(\phi'(\rho)-s\phi(\rho))ps\coth(\kappa\rho)\right]e^{-s\rho}\\
            &=\left[s(p\coth(\kappa \rho)-2)\phi'(\rho)+s^2(1-p\coth(\kappa\rho))\phi(\rho)\right]e^{-s\rho}.
        \end{align*}

        On the one hand, using a polar coordinate transform and the second branch of \eqref{eq:truncation:1} we have
        \[
            \int_\Omega |u_\delta|^p\diff v_g=\int_\frac{\delta}{2}^\delta \phi(t)^pe^{-pst}\frac{\sinh^{n-1}(
            \kappa t)}{\kappa^{n-1}}\diff t\ge \frac{1}{\kappa^{n-1}} \int_{\frac{\delta}{2}+1}^{\delta-1}e^{-pst}\sinh^{n-1}(\kappa t)\diff t.    
        \]
        Observe that 
        \[e^{-pst}\sinh^{n-1}(\kappa t)=(e^{-\kappa t}\sinh(\kappa t))^{n-1}=\left(\frac{1}{2}-\frac{e^{-2\kappa t}}{2}\right)^{n-1}\]
        is strictly increasing in \(t\), thus we have the following estimate
        \[
            \int_\Omega |u_\delta|^p\diff v_g\ge \frac{1}{\kappa^{n-1}}\left(\frac{\delta}{2}-2\right)\left(\frac{1}{2}-\frac{e^{-\kappa \delta-2\kappa}}{2}\right)^{n-1}:=E_1(\delta).
        \]

        On the other hand, similarly to the previous computations,  one has
        \[
            \int_\Omega |\Delta_g u_\delta|^p\diff v_g= \int_\frac{\delta}{2}^\delta \frac{1}{\kappa^{n-1}}\left|s(p\coth(\kappa t)-2)\phi'(t)+s^2(1-p\coth(\kappa t))\phi(t)\right|^p\left(e^{-\kappa t}\sinh(\kappa t)\right)^{n-1}\diff t.
        \]
        Observe that
        $
            \Phi(t)=\left|s(p\coth(\kappa t)-2)\phi'(t)+s^2(1-p\coth(\kappa t))\phi(t)\right|^p\left(e^{-\kappa t}\sinh(\kappa t)\right)^{n-1}
       $
        is bounded. Let \(M_1\) and \(M_2\) be the maximum of \(\Phi(t)\) on \(\left[\frac{\delta}{2},\frac{\delta}{2}+1\right]\) and \(\left[\delta-1,\delta\right]\), respectively. Thus using again \eqref{eq:truncation:1} we have 
        \[
            \int_\Omega |\Delta_g u_\delta|^p\diff v_g\le M_1+M_2+\frac{s^{2p}}{\kappa^{n-1}}\int_{\frac{\delta}{2}+1}^{\delta-1} (p\coth(\kappa t)-1)^p\left(e^{-\kappa t}\sinh(\kappa t)\right)^{n-1}\diff t.
        \]
        Since \((p\coth(\kappa t)-1)^p\) is decreasing, \(e^{-\kappa t}\sinh(\kappa t)\) is increasing, and both expressions are positive, we get 
        \begin{equation}\label{eq:estimate2:proof1}
            \int_\Omega |\Delta_g u_\delta|^p\diff v_g\le M_1+M_2+\frac{s^{2p}}{\kappa^{n-1}}\left(\frac{\delta}{2}-2\right)\left(p\coth\left(\left(\frac{\delta}{2}+1\right)\kappa\right)-1\right)^p\left(\frac{1}{2}-\frac{e^{-2\kappa (\delta-1)}}{2}\right)^{n-1}:=E_2(\delta).
        \end{equation}
        Using the two estimates, we have 
        \[
             \lim_{\delta\to\infty}\frac{\displaystyle\int_\Omega|\Delta_g u_\delta|^p\diff v_g}{\displaystyle\int_\Omega|u_\delta|^p\diff v_g}\le \lim_{\delta\to\infty}\frac{E_2(\delta)}{E_1(\delta)}=s^{2p}(p-1)^p=\left(\frac{(n-1)^2(p-1)\kappa^2}{p^2}\right)^p,
        \]
        hence the inequality is sharp.
\hfill $\square$

%    \begin{remark} Observe that when \(p=2\), then the best constant is \[C=\frac{(n-1)^4\kappa^4}{16}.\]
%    \end{remark}

    \subsection{Buckling plate problem: proof of Theorem \ref{thm:bpp:intro}}\label{ssec:buckling} 
%    For the buckling problem the sharp spectral gap estimate is presented in . For the sake of readability we repeat the statement here, before the proof.
%    \begin{theorem}\label{thm:bpp}
%        Let \((M,g)\) be an \(n\)-dimensional, complete, non-compact Riemannian manifold, with \(n\ge 2\). Suppose that the sectional curvature satisfies \({\bf K}\ge-\kappa
%        ^2\), for some \(\kappa>0\). Let \(\Omega\subseteq M\) be a domain. Then for every \(u\in C_0^\infty(\Omega)\) one has
%        \begin{equation}
%            \int_\Omega |\Delta_g u|^2\diff v_g\ge \frac{(n-1)^2\kappa^2}{4}\int_\Omega|\nabla_g u|^2.
%            \label{eq:mcKeangru}
%        \end{equation}
%        Moreover the constant is sharp.
%    \end{theorem}
%    \begin{proof}
        In Theorem \ref{thm:maingru} let us choose 
        \[
            L\equiv(n-1)\kappa,\quad W\equiv C,\quad G\equiv a\quad\mbox{and}\quad H\equiv b,   
        \]
        for some constants \(C,a,b>0\), which will be determined later. Condition \ref{cond:c1p} clearly holds. By Laplace comparison (see Theorem \ref{comparison-theorem}) one has 
        \[\Delta_g \rho\ge (n-1)\kappa\coth(\kappa \rho)\ge (n-1)\kappa,\]
        hence condition \ref{cond:c2p} holds as well. Inequality \eqref{eq:pdigen} from condition \ref{cond:c3p} is equivalent to   
        \[
            Cb(n-1)\kappa-Cb^2\ge a^2.
        \]
        Provided that the above inequality holds Theorem \ref{thm:maingru} implies 
        \[
            \int_\Omega |\Delta_g u|^2\diff v_g\ge (2a-C)\int_\Omega|\nabla_g u|^2 \diff v_g,\qquad\forall u\in C_0^\infty(\Omega).
        \]
        To obtain the best spectral gap estimate we need to maximize
        \[
            f(b,C):=2\sqrt{Cb(n-1)\kappa-Cb^2}-C.
        \]
        A simple computation implies that 
        \[
            \max_{b,C} f(b,C)=f\left(\frac{(n-1)\kappa}{2},\frac{(n-1)^2\kappa^2}{4}\right)=\frac{(n-1)^2\kappa^2}{4}, 
        \]
        which implies precisely \eqref{eq:bppvar}.
    
        The proof of the sharpness is similar as before. Fix \(\delta>0\), denote \(s=\frac{(n-1)\kappa}{2}\) and define 
        \[
            u_\delta=\phi(\rho) e^{-s\rho} 
        \] on \(\Omega={\bf M}_{-\kappa^2}^n\) where \(\phi\) is the truncation function from \eqref{eq:truncation:1}. Since 
        \[
            \nabla_g u_\delta=(\phi'(\rho)-s\phi(\rho))e^{-s\rho}\nabla_g\rho,
        \]
        using a polar coordinate transform and the second branch of \eqref{eq:truncation:1} we have
        \begin{align*}
            \int_\Omega |\nabla_g u_\delta|^2\diff v_g&=\int_\frac{\delta}{2}^\delta (\phi'(t)-s\phi(t))^2e^{-2st}\frac{\sinh^{n-1}(
            \kappa t)}{\kappa^{n-1}}\diff t\ge \frac{s^2}{\kappa^{n-1}} \int_{\frac{\delta}{2}+1}^{\delta-1}e^{-2st}\sinh^{n-1}(\kappa t)\diff t\\
            &=\frac{s^2}{\kappa^{n-1}} \int_{\frac{\delta}{2}+1}^{\delta-1}\left(\frac{1}{2}-\frac{e^{-2\kappa t}}{2}\right)^{n-1}\diff t \\
            &\ge \frac{s^2}{\kappa^{n-1}}\left(\frac{\delta}{2}-2\right)\left(\frac{1}{2}-\frac{e^{-\kappa \delta-2\kappa}}{2}\right)^{n-1}:=E_1(\delta).
        \end{align*}
Recall the estimate \eqref{eq:estimate2:proof1} for \(p=2\) to obtain
        \[
            \int_\Omega |\Delta_g u_\delta|^2\diff v_g\le M_1+M_2+\frac{s^{2p}}{\kappa^{n-1}}\left(\frac{\delta}{2}-2\right)\left(p\coth\left(\left(\frac{\delta}{2}+1\right)\kappa\right)-1\right)^p\left(\frac{1}{2}-\frac{e^{-2\kappa (\delta-1)}}{2}\right)^{n-1}:=E_2(\delta),
        \]
        where \(M_1\) and \(M_2\) is the maximum of the bounded function 
        \[
            \Phi(t)=\left(s(2\coth(\kappa t)-2)\phi'(t)+s^2(1-2\coth(\kappa t))\phi(t)\right)^2\left(e^{-\kappa t}\sinh(\kappa t)\right)^{n-1}, 
        \]
        on the intervals \(\left[\frac{\delta}{2},\frac{\delta}{2}+1\right]\) and \(\left[\delta-1,\delta\right]\), respectively.

        Using the two estimates, we have 
        \[
             \lim_{\delta\to\infty}\frac{\displaystyle\int_\Omega|\Delta_g u_\delta|^2\diff v_g}{\displaystyle\int_\Omega|\nabla_g u_\delta|^2\diff v_g}\le \lim_{\delta\to\infty}\frac{E_2(\delta)}{E_1(\delta)}=s^2=\frac{(n-1)^2\kappa^2}{4},
        \]
        hence the inequality is sharp.
    \hfill $\square$

    \subsection{Higher-order estimates}\label{ssec:hogaps} We conclude this section by presenting some higher-order estimates concerning both problems from the previous subsections. In case of the clamped plate problem, the following higher-order estimates hold.
    \begin{theorem}\label{thm:ho:cp}
        Let \((M,g)\) be an \(n\)-dimensional Cartan-Hadamard manifold  as in Theorem \ref{thm:cpp:intro}. 
        Let \(\Omega\subset M\) be a domain and \(p>1\). Then for every \(u\in C_0^\infty(\Omega)\) and \(k\ge1\) one has
        \begin{align}
            \label{eq:homc1}\int_\Omega |\Delta_g^k u|^p\diff v_g&\ge \left(\frac{(n-1)^2(p-1)\kappa^2}{p^2}\right)^{kp}\int_\Omega|u|^p\diff v_g,\\
            \label{eq:homc2}\int_\Omega |\nabla_g \Delta_g^k u|^p\diff v_g&\ge \left(\frac{(n-1)\kappa}{p}\right)^p\left(\frac{(n-1)^2(p-1)\kappa^2}{p^2}\right)^{kp}\int_\Omega|u|^p\diff v_g.
        \end{align}
        Moreover, the constants in \eqref{eq:homc1} and \eqref{eq:homc2} are sharp.
    \end{theorem}
    \begin{proof}
        
         Inequality \eqref{eq:homc1} can be obtained by iterative applications of Theorem \ref{thm:cpp:intro} for the functions \(u:=\Delta^l u\) for all \(l\in\{0,1,\dots,k-1\}\). To obtain \eqref{eq:homc2}, apply inequality \eqref{eq:fmpvar}  as well for the function \(u:=\Delta^ku\). 
        
        To prove the sharpness,  let \(s=\frac{(n-1)\kappa}{p}\) as before, and choose $u_\delta = \phi(\rho)e^{-s\rho}$ on $\Omega={\bf M}_{-\kappa^2}^n,$
        where \(\phi\) is the truncation function from \eqref{eq:truncation:1}. For simplicity, let us denote \[ L(\delta )=(n-1)\kappa\coth(\kappa\delta )=ps\coth(\kappa\delta ).\]
        To obtain the proof for general \(k\ge 1\) we have to compute \(\Delta_g^k u_\delta\) and give an appropriate lower bound for it. This computation becomes more and more involved for higher values of \(k\);  however, based on the ideas used in case \(k=1\), we can  significantly simplify them.   
        
        The first observation is that the branches when \(t\in[\frac{\delta}{2},\frac{\delta}{2}+1]\) and \(t\in[\delta-1,\delta]\) do not have any contribution to the final limit. This is due to the fact that the integrands are bounded and the integration interval is of unit length, hence these integrals are dominated by the leading term provided by the branch when \(t\in[\frac{\delta}{2}+1,\delta-1]\). The same phenomenon occurs when \(k\ge1\). We can restrict our attention only to this case, and technically we can assume in the sequel that \(\phi=1\).
        
        The second observation is that since \(u\) is radially symmetric, we have 
        $\Delta_g u_\delta = u_\delta''+Lu_\delta'=s(s-L)e^{-s\rho}.$
        Based on the computation for the case \(k=1\), we are only interested in the asymptotic behavior when \(\delta\to\infty\). One can easily verify that the \(k\)-th derivatives of \(L\) satisfy
        \[
        \lim_{\delta\to\infty}L^{(k)}(\delta)=\begin{cases}
            ps, &\mbox{if } k=0,\\
            0,&\mbox{if } k\ge 1.
        \end{cases}
        \] Using this fact, for the bi-laplacian one has
        \begin{align*}
            \Delta_g^2 u_\delta&= u_\delta^{(4)} + 2Lu_\delta^{(3)}+L^2u_\delta''+L''u_\delta'+2L'u_\delta''+LL'u_\delta'\\
            &\sim u_\delta^{(4)} + 2Lu_\delta^{(3)}+L^2u_\delta''=s^2(s-L)^2e^{-s\rho}\sim s^4(1-p)^2e^{s\rho}.
        \end{align*}
        By similar argument for general \(k\ge 1\) one has 
        \begin{equation}
            \label{eq:asimpt1}
            \Delta_g^k u_\delta\sim s^{2k}(1-p)^ke^{-s\rho}.
        \end{equation}
        Using the estimates for \(\displaystyle\int_\Omega |u_\delta|^p \diff v_g\), from the proof of the case when \(k=1\) we obtain 
        \[
        \lim_{\delta\to\infty}\frac{\displaystyle \displaystyle \int_\Omega|\Delta_g^k u_\delta|^p \diff v_g}{\displaystyle \int_\Omega |u_\delta|^p \diff v_g} = s^{2kp}(1-p)^{kp}=\left(\frac{(n-1)^2(p-1)\kappa^2}{p^2}\right)^{kp}.  
        \]
        Taking the gradient in relation \eqref{eq:asimpt1} implies 
        $\nabla_g\Delta_g^k u_\delta\sim -s\cdot s^{2k}(1-p)^ke^{-s\rho},$
        hence we obtain 
        \[
        \lim_{\delta\to\infty}\frac{\displaystyle \int_\Omega|\nabla_g\Delta_g^k u_\delta|^p \diff v_g}{\displaystyle \int_\Omega |u_\delta|^p \diff v_g} = s^p s^{2kp}(1-p)^{kp}=\left(\frac{(n-1)\kappa}{p}\right)^p\left(\frac{(n-1)^2(p-1)\kappa^2}{p^2}\right)^{kp},  
        \]
     which concludes the proof.    
    \end{proof}
    
    In case of the buckling plate problem, the following higher-order estimates hold.
    \begin{theorem}\label{thm:ho:bp}
        Let \((M,g)\) be an \(n\)-dimensional Cartan-Hadamard manifold  as in Theorem \ref{thm:bpp:intro}. Let \(\Omega\subset M\) be a domain. Then for every \(u\in C_0^\infty(\Omega)\) and \(k\ge 1\) one has
        \begin{align}
            \label{eq:homc1buc}\int_\Omega |\Delta_g^k u|^2\diff v_g&\ge \left(\frac{(n-1)\kappa}{2}\right)^{4k-2}\int_\Omega|\nabla_g u|^2\diff v_g,\\
            \label{eq:homc2buc}\int_\Omega |\nabla_g \Delta_g^k u|^2\diff v_g&\ge \left(\frac{(n-1)\kappa}{2}\right)^{4k}\int_\Omega|\nabla_gu|^2\diff v_g.
        \end{align}
    \end{theorem}

    \begin{proof} Inequality \eqref{eq:homc1buc} can be obtained as follows. First, use inequality \eqref{eq:homc1} for \(p=2\), \(u:=\Delta_g u\) and \(k:=k-1\); next, we have to apply Theorem \ref{thm:bpp:intro}. To obtain inequality \eqref{eq:homc2buc}, apply the spectral gap estimate of the fixed membrane problem \eqref{eq:fmpvar} as well, for the function \(u:=\Delta^ku\).  
        
        The sharpness can be proven similarly as before: choose \(u_\delta = \phi(\rho)e^{-s\rho}\) with \(s=\frac{(n-1)\kappa}{p}\) and \(\phi\) is the truncation function;  letting \(\delta\to \infty\), we  obtain the desired result.
    \end{proof}

    \section{Byproducts: Sharp Rellich inequalities}\label{sec:Rellich}
    This section  is devoted to  applications of our general functional inequalities to obtain various Rellich inequalities on Cartan-Hadamard manifolds. First,  we use Theorem \ref{thm:mainu} to extend the classical, weighted Rellich inequalities to Cartan-Hadamard manifolds. Next, based on these results, we state higher-order Rellich inequalities. Finally, we present short proofs to some, formally well-known Rellich-type inequalities, highlighting further applicability  of Theorems \ref{thm:mainu} \& \ref{thm:maingru}.

    \subsection{Classical and weighted Rellich inequalities}  
     The weighted Rellich inequality reads as follows;  see Mitidieri \cite[Theorem 3.1]{mitidieri2000simple} for the Euclidean version.
    \begin{theorem}
        \label{thm:Rellich:weighted}
        Let \((M,g)\) be an \(n\)-dimensional Cartan-Hadamard manifold with \(n\ge 5\),   \(\Omega\subset M\) be a domain and \(p,\gamma\in\mathbb{R}\) such that \[1<p<n/2\quad\mbox{and}\quad 2-\frac{n}{p}<\gamma<\frac{n(p-1)}{p}.\] Fix \(x_0\in \Omega\) and let \(\rho=d_{x_0}\). Then for every \(u\in C_0^\infty(\Omega)\) one has
        \begin{equation}
            \int_\Omega \rho^{\gamma p} |\Delta_g u|^p\diff v_g\ge \left(\frac{n}{p}-2+\gamma\right)^p\left(\frac{n(p-1)}{p}-\gamma\right)^p\int_\Omega\frac{|u|^p}{\rho^{(2-\gamma)p}}\diff v_g,
            \label{eq:weightedRellich}
        \end{equation}
    and the constant in \eqref{eq:weightedRellich} is sharp. 
    \end{theorem}
    \begin{proof} 
        In Theorem \ref{thm:mainu} we choose
        \[
            \quad L(t)=\frac{n-1}{t},\quad W(t)=\frac{C}{t^{(2-\gamma)p}},\quad  w=t^{\gamma p},\quad G(t)=\frac{a}{t^{2p-2}}\quad\mbox{and}\quad H(t)=\frac{b}{t},\qquad\forall t\in(0,\sup_\Omega\rho),   
        \]
        and for some constant \(C,a,b>0\) which will be determined later. Condition \ref{cond:c1} clearly holds, while a straightforward computation and the Laplace comparison (see Theorem \ref{comparison-theorem}) implies condition \ref{cond:c2} as well. Inequality \eqref{eq:odigen} from condition \ref{cond:c3} is equivalent to 
        \[
            f(a,b):=a (p-1) (2 (b + 1) n - (2 + b)^2 p) + 
            ap (2 b (p-1) + 4 p-n-2) \gamma - ap^2 \gamma^2 -(p-1)a^{\frac{p}{p-1}}\ge C.
        \] 
                The best value for the constant $C$ is obtained as
        \[
            \max_{a,b}f(a,b)=f\left(\left(\frac{n}{p}-2+\gamma\right)^{p-1}\left(\frac{n(p-1)}{p}-\gamma\right)^{p-1},\frac{n}{p}-2+\gamma\right)=\left(\frac{n}{p}-2+\gamma\right)^p\left(\frac{n(p-1)}{p}-\gamma\right)^p,
        \]
    which provides the desired inequality \eqref{eq:weightedRellich}.     The sharpness of this constant can be verified in a similar way as in the proof of Theorem \ref{thm:cpp:intro} by using suitable truncation functions. 
%        In order to obtain the best constant \(C\) we need to maximize \(f\), which is obtain for the pair` $(\widetilde{a},\widetilde{b})$, where    
%        \[
%           \widetilde{a}=\left(\frac{n}{p}-2+\gamma\right)^{p-1}\left(\frac{n(p-1)}{p}-\gamma\right)^{p-1} \quad\mbox{and}\quad \widetilde{b}=\frac{n}{p}-2+\gamma.\]
%        A simple computation then yields the desired expression  
%        \[
%            \max_{a,b}f(a,b)=f(\widetilde{a},\widetilde{b})=\left(\frac{n}{p}-2+\gamma\right)^p\left(\frac{n(p-1)}{p}-\gamma\right)^p,
%        \]
    \end{proof}
    \begin{remark}\label{cor:classicRellich}\rm 
        Choosing \(\gamma=0\) in Theorem \ref{thm:Rellich:weighted}, it yields that
        \[
            \int_\Omega |\Delta_g u|^p\diff v_g\ge \left(\frac{n}{p}-2\right)^p\left(\frac{n(p-1)}{p}\right)^p\int_\Omega\frac{|u|^p}{\rho^{2p}}\diff v_g,\qquad\forall u\in C_0^\infty(\Omega).
        \]
        In particular, for  \(p=2\) one has
        \[
            \int_\Omega |\Delta_g u|^2\diff v_g\ge \left(\frac{n^2-4n}{4}\right)^2\int_\Omega\frac{|u|^2}{\rho^{4}}\diff v_g,\qquad\forall u\in C_0^\infty(\Omega).
        \]
    \end{remark}
    \subsection{Higher-order Rellich inequalities} In order to obtain higher-order Rellich inequalities, we iteratively apply Theorem \ref{thm:Rellich:weighted} and use the sharp Hardy inequality
    \begin{equation}\label{eq:sharp:hardy}
        \int_\Omega |\nabla_g u|^p\diff v_g\ge \left(\frac{n-p}{p}\right)^p \int_\Omega |u|^p \diff v_g,\qquad\forall u\in C_0^\infty(\Omega),
    \end{equation}
    where $1<p<n$ see e.g.\ Kajántó, Kristály, Peter and Zhao \cite{ricattipair2023}.
    We have the following extension of Mitidieri \cite[Theorem 3.3]{mitidieri2000simple}.
    
    \begin{theorem}\label{thm:Rellich:ho}  Let \((M,g)\) be an \(n\)-dimensional Cartan-Hadamard manifold with \(n\ge 5\),   \(\Omega\subset M\) be a domain, fix \(x_0\in\Omega\) and define \(\rho=d_{x_0}\). We have the following inequalities$:$
    \begin{enumerate}[label={\rm (\roman*)}]
        \item If \(k\ge 1\) and \(n>2kp\), then 
        \[
            \int_\Omega |\Delta_g^k u|^p \diff v_g\ge \Lambda_{\text{\rm r},1}(k,p) \int_\Omega \frac{|u|^p}{\rho^{2kp}}\diff v_g,\quad\forall u\in C_0^\infty(\Omega),\]
        where the sharp constant is 
        \[
            \Lambda_{\text{\rm r},1}(k,p)=\prod_{s=1}^k\left(\frac{n}{p}-2s\right)^p\left(\frac{n(p-1)}{p}+2s-2\right)^p.  
        \]
        \item If \(k\ge 1\) and \(n>(2k+1)p\), then
        \[
            \int_\Omega |\nabla_g\Delta_g^k u|^p \diff v_g\ge \Lambda_{\text{\rm r},2}(k,p) \int_\Omega \frac{|u|^p}{\rho^{(2k+1)p}}\diff v_g,\quad\forall u\in C_0^\infty(\Omega)
        \]
        where the sharp constant is 
        \[
            \Lambda_{\text{\rm r},2}(k,p)=\left(\frac{n-p}{p}\right)^p\prod_{s=1}^k\left(\frac{n}{p}-2s-1\right)^p\left(\frac{n(p-1)}{p}+2s-1\right)^p.      
        \]
    \end{enumerate}
    \end{theorem}
    \begin{proof} To obtain (i), we apply Theorem \ref{thm:Rellich:weighted} for every \(s\in\{1,\dots,k\}\) with the choices \(u:= \Delta_g^{k-s} u\) and \(\gamma:=2-2s\). To prove (ii), first apply Theorem \ref{thm:Rellich:weighted} for every \(s\in\{1,\dots,k\}\) by choosing \(u:=\Delta_g^{k-s} u\) and \(\gamma\to1-2s\), and then use inequality \eqref{eq:sharp:hardy} with the choice \(u:= \Delta_g^k u\). 
  The sharpness of \(\Lambda_{\text{\rm r},1}(k,p)\) and \(\Lambda_{\text{\rm r},2}(k,p)\) can be proved is the usual manner. 
    \end{proof}

    \subsection{Further applications} We conclude the paper by showing further applicability of our general functional inequalities to produce short proofs for some improved Rellich-type inequalities.

%    The first result is an applications of the first functional inequality on manifolds with sectional curvature \({\bf K}\le 0\). It can be stated as follows.
    \begin{theorem}\label{thm:impr:1}
        Let \((M,g)\) be an \(n\)-dimensional Cartan-Hadamard manifold with \(n\ge 5\),   \(\Omega\subset M\) be a ball centered at \(x_0\in M\) with unit radius. Define \(\rho=d_{x_0}\). Then for every \(u\in C_0^\infty(\Omega)\) one has 
        \[\int_\Omega|\Delta_g u|^2 \diff v_g\ge \frac{n^2(n-4)^4}{16}\int_\Omega\frac{u^2}{\rho^4}\diff v_g+\frac{n(n-4)j_{0,1}^2}{2}\int_\Omega \frac{u^2}{\rho^2}\diff v_g.\]
%        where \(j_{0,1}\) denotes the first positive zero of the Bessel function of the first kind with order zero.
    \end{theorem}
    \begin{proof}
        Apply Theorem \ref{thm:mainu} for $p=2$ and with the following chooses: 
        \[L(t)=\frac{n-1}{t},\quad G(t)=\frac{n(n-4)}{4t^2}\quad\mbox{and}\quad H(t)=\frac{n-4}{2t}+\frac{j_{0,1}\cdot J_1(j_{0,1}t)}{J_0(j_{0,1}t)},\qquad \forall t\in (0,1);\]
        A simple computation yields the desired inequality. 
    \end{proof}

\begin{remark}\rm We notice that the leading constant $\frac{n^2(n-4)^4}{16}$ in Theorem \ref{thm:impr:1} is sharp. 
	
\end{remark}

    The second result deals with the case when \({\bf K}\le-\kappa^2\) for some \(\kappa>0\), and can be formulated as follows.
    \begin{theorem}\label{thm:impr:2}
        Let \((M,g)\) be an \(n\)-dimensional Cartan-Hadamard manifold with \(n\ge 5\), and  the sectional curvature satisfies \({\bf K}\le-\kappa^2\) for some \(\kappa>0\). Let \(\Omega\subset M\) be a domain, fix \(x_0\in\Omega\) and define \(\rho=d_{x_0}\). Then for every \(u\in C_0^\infty(\Omega)\) one has 
        \[\int_\Omega|\Delta_g u|^2\diff v_g\ge \frac{(n-1)^4\kappa^4}{16}\int_\Omega u^2\diff v_g+\frac{(n-1)^2\kappa^2}{8}\int_\Omega \frac{u^2}{\rho^2}\diff v_g+\frac{(n-1)^3(n-3)\kappa^4}{8}\int_\Omega\frac{u^2}{\sinh^2(\kappa\rho)}\diff v_g.\]
    \end{theorem}
    \begin{proof}
        We apply Theorem \ref{thm:mainu} by choosing 
        \[L(t)=(n-1)\kappa\coth(\kappa t),\quad G(t)=\frac{(n-1)^2\kappa^2}{4}\quad\mbox{and}\quad H(t)=\frac{(n-1)\kappa\coth(\kappa t)}{2}-\frac{1}{2t},\qquad \forall t>0.\]
        The required  inequality follows after a simply computation. 
    \end{proof}

\begin{remark}\rm 
The inequality from Theorem \ref{thm:impr:2} can be compared with the main results from	Berchio,  Ganguly and Roychowdhury \cite{Berchio-etal}, where the authors established various Rellich-type identities, which imply in turn sharp Rellich-type improvements on the hyperbolic space. 
\end{remark}

    The third result is a simple application of Theorem \ref{thm:maingru}.
    \begin{theorem}\label{thm:impr:3}
    	Let \((M,g)\) be an \(n\)-dimensional Cartan-Hadamard manifold with \(n\ge 8\),   \(\Omega\subset M\) be a domain, \(x_0\in M\) be fixed and define \(\rho=d_{x_0}\).  Then for every \(u\in C_0^\infty(\Omega)\) one has 
        \[\int_\Omega|\Delta_gu|^2 \diff v_g\ge \frac{n^2}{4}\int_\Omega \frac{|\nabla_g u|^2}{\rho^2}\diff v_g,\]
        and the constant $\frac{n^2}{4}$ is sharp. 
    \end{theorem}
    
%    \textcolor{red}{A 3.2 T\'etel szerint a $W>0 \Leftrightarrow n\geq 9$.}
    \begin{proof}
        We apply Theorem \ref{thm:maingru} with the  choices
        \[
            L(t)=\frac{n-1}{t},\quad G(t)=\frac{n(n-4)}{4t^2},
        \quad H(t)=\frac{n-4}{2t} \quad\mbox{and}\quad W(t)=\frac{n(n-8)}{4},\qquad \forall t>0,    
        \]
       which provides the proof. For the sharpness of  $\frac{n^2}{4}$ we may proceed as in the proof of Theorem  \ref{thm:cpp:intro}.
    \end{proof}
    
    \begin{remark}\rm  Note that Theorem \ref{thm:impr:3} is expected to hold for every $n\geq 5$; however, the technical condition $n\ge 8$ is required to guarantee  the applicability of  Theorem \ref{thm:maingru} ($W>0$ whenever $n\ge 9$, and if $n=8$ then $W=0$, in which case the proof of Theorem \ref{thm:maingru} is obvious). A similar restriction also appeared in the Finsler context by proving quantitative Rellich inequalities, see Krist\'aly and Repovš \cite{Kristaly-Repovs}, where another approach were applied.  
    \end{remark}

%    \bibliographystyle{plain}
%    \bibliography{references}

\end{document}